\documentclass[11pt,a4paper,reqno,english]{amsart}

\usepackage[utf8]{inputenc}
\usepackage[T1]{fontenc}
\usepackage[osf]{mathpazo}
\AtBeginDocument{
  \DeclareSymbolFont{AMSb}{U}{msb}{m}{n}
  \DeclareSymbolFontAlphabet{\mathbb}{AMSb}}

\usepackage{amsmath,amsthm,amssymb,amscd,mathrsfs}
\usepackage[dvipsnames]{xcolor}
\usepackage{tikz}
\usepackage{blkarray}
\usepackage{url}
\usepackage{graphicx}
\usepackage{setspace}

\usepackage[breaklinks,colorlinks,pagebackref]{hyperref} 

\definecolor{Red}{RGB}{160, 4, 23}
\definecolor{Green}{RGB}{27, 132, 5}
\definecolor{Blue}{RGB}{11, 78, 188}

\hypersetup{linkcolor=Red,citecolor=Green}
\usepackage[capitalise,nameinlink]{cleveref}
\crefformat{equation}{(#2#1#3)}

\usepackage[msc-links,lite,alphabetic]{amsrefs}
\def\MR#1{%
    \relax\ifhmode\unskip\spacefactor3000 \space\fi
\href{http://www.ams.org/mathscinet-getitem?mr=#1}{ MR#1}
}

\makeatletter
\renewcommand{\BibLabel}{%
    \Hy@raisedlink{\hyper@anchorstart{cite.\CurrentBib}\hyper@anchorend}%
    [\thebib]%
}
\makeatother

\usepackage{mathscinet}

\usepackage{soul, xcolor}
\setstcolor{red}

\newtheorem{theorem}{Theorem}[section]
\newtheorem{lemma}[theorem]{Lemma}
\newtheorem{proposition}[theorem]{Proposition}
\newtheorem{corollary}[theorem]{Corollary}

\theoremstyle{definition}
\newtheorem{remark}[theorem]{Remark}
\newtheorem{definition}[theorem]{Definition}

\newtheorem{question}[theorem]{Question}

\newcommand{\e}{\mathbf{e}}
\newcommand{\SC}{S}
\newcommand{\K}{\mathbb{K}}
\newcommand{\QQ}{\mathbb{Q}}
\newcommand{\RR}{\mathbb{R}}
\newcommand{\ZZ}{\mathbb{Z}}
\newcommand{\CC}{\mathbb{C}}
\newcommand{\G}{\mathcal{G}}
\newcommand{\1}{\mathbf{1}}
\newcommand{\p}{\mathfrak{p}}
\newcommand{\C}{\mathcal{C}}
\newcommand{\x}{\mathbf{x}}
\newcommand{\bchi}{\chi_{_\mathrm{Bor}}}

\newcommand{\A}{\mathcal{A}}
\newcommand{\B}{\mathcal{B}}
\newcommand{\GL}{\mathrm{GL}}
\newcommand{\SL}{\mathrm{SL}}
\newcommand{\z}{\mathbf{z}}
\newcommand{\CG}{\mathcal{G}}

\renewcommand{\d}[1]{\mathrm{d}#1}
\DeclareMathOperator{\Cay}{Cay}
\newcommand{\bq}{\mathit{q}}
\newcommand{\KK}{\mathcal{K}}

\DeclareMathOperator{\op}{op}
\usepackage[top=2cm,bottom=2cm,right=1.9cm,left=1.9cm]{geometry}

\title[ Sets of positive upper density over local fields]{Polynomial configurations in sets of positive upper density over local fields}

\author[ M. Bardestani]{Mohammad Bardestani}
\address{DPMMS, Centre for Mathematical Sciences, Wilberforce Road, Cambridge, CB3 0WB.}
\email{ mohammad.bardestani@gmail.com}

\author[ K. Mallahi-Karai]{Keivan Mallahi-Karai}
\address{Jacobs University Bremen, Campus Ring I, 28759 Bremen, Germany.}
\email{k.mallahikarai@jacobs-university.de }

\begin{document}


\begingroup
\def\uppercasenonmath#1{} 
\maketitle
\endgroup

{
\footnotesize
\hypersetup{linkcolor=Blue}
\tableofcontents
}

\begin{abstract}
Let  $F(x)=(f_1(x), \dots, f_m(x))$  be such that $1, f_1, \dots, f_m$ are linearly independent polynomials with real coefficients. Based on 
ideas of Bachoc, DeCorte, de Oliveira and Vallentin in combination with estimating certain oscillatory integrals with polynomial phase we will 
show that the independence ratio of the Cayley graph of $\RR^m$ with respect to the portion of the graph of $F$ 
defined by $a\leq \log |s| \leq T$ is at most $O(1/(T-a))$. We conclude that 
if $I \subseteq \RR^m$ has positive upper density, then the difference set $I-I$ contains
vectors of the form $F(s)$ for an unbounded set of values $s \in \RR$. It follows that the Borel 
chromatic number of the Cayley graph of $\RR^m$ with respect to the set $\{ \pm F(s): s \in \RR \}$
is infinite. Analogous results are also proven when $\RR$ is replaced by the field of $p$-adic numbers
$\QQ_p$. At the end, we will also show the existence of real analytic functions $f_1, \dots, f_m$, for which the analogous statements no longer hold.
\end{abstract}
\section{Introduction}  The purpose of this paper is to establish a relation between the theory of real (and $p$-adic) singular integrals to the Borel chromatic number and independence ratio of Cayley graphs of $\RR^m$ (and $\QQ_p^m$) with respect to sets defined by algebraic equations. 

For an abelian group $\Gamma$ (in this paper, typically, $\RR^m$ or $\QQ_p^m$) and a symmetric subset $S \subseteq \Gamma$ (i.e., $S=-S$) which does not contain the identity element of $\Gamma$, let $\Cay(\Gamma,S)$ denote the Cayley graph of $\Gamma$ with respect to $S$. The vertex set of $\Cay(\Gamma,S)$ is identified with $\Gamma$ and vertices $x,y\in \Gamma$ are declared adjacent whenever $x-y\in S$.
Recall that the {\bf chromatic number} of a graph $\CG$, denoted by $\chi(\CG)$, is the least cardinal $c$ such that the vertices of $\CG$ can be partitioned into $c$ sets (called {\bf color classes}) such that no color class contains an edge in $\CG$. When the vertex set of $\CG$ is equipped with a topology, the {\bf Borel chromatic number} of $\CG$, denoted by $\bchi(\CG)$, is defined to be the least cardinal $c$ such that the vertex set of $\CG$ can be partitioned into $c$ Borel subsets none of which contains two adjacent vertices of $\CG$.  

It is easy to observe that when $S$ is a compact symmetric subset of $\RR^m$ (or $\QQ_p^m$) not containing the origin, then the Borel (and, {\it a fortiori}, the ordinary) chromatic number of $\Cay(\RR^m, S)$ is finite. Still, determining the exact values of the ordinary or Borel chromatic number could be a challenging problem. 
Perhaps the most prominent instances of such graphs are the {\bf unit distance graphs}, defined to be the Cayley graph of $\RR^m$ (for some integer $m \ge 2$) with respect to the unit sphere 
$${\bf S}^{m-1}=\left\{(x_1,\dots,x_m)\in\RR^m: x_1^2+\dots+x_m^2=1\right\}.$$
Even for the special case $m=2$, the elementary inequalities $4\leq \chi(\Cay(\RR^2,{\bf S}^1))\leq 7$ remained the best known bounds for many years and only quite recently the lower bound was improved to $5$ in the spectacular work of de Grey~\cite{Grey}. Frankl and Wilson~\cite{Frankl-Wilson} answered affirmatively a question of Erd\H{o}s  by proving that for $m \ge 3$ the lower bound  $c_1^m\leq \chi(\Cay(\RR^m,{\bf S}^{m-1}))$ holds for $c_1=1.207+o(1)$.  This bound was subsequently improved to $c_1=1.239+o(1)$ by Raigorodskii~\cite{Raigorodskii}. Larman and Rogers~\cite{Larman-Rogers} established the upper bound $\chi(\Cay(\RR^m,{\bf S}^{m-1}))\leq c_2^m$, for $c_2=3+o(1)$.  For more results on this problem, we refer the reader to~\cites{Soifer, Szekely} and the references therein.

Falconer~\cite{Falconer1} initiated the study of the Borel chromatic number of the unit distance graphs by using ideas from geometric measure theory and showed that
$$\bchi(\Cay(\RR^m,{\bf S}^{m-1})) \ge m+3.$$
More recently, Bachoc, Nebe, de Oliveira and Vallentin~\cite{Bachoc} employed tools from 
Fourier analysis and linear programming methods to improve the existing lower bounds for $\bchi(\Cay(\RR^m,{\bf S}^{m-1}))$ when $10\leq m\leq 24$. These bounds have been further extended by de Oliveira and Vallentin~\cite{Vallentin} to the range $3\leq m\leq 24$.

The definition of the unit distance graphs can be reformulated via quadratic forms over an arbitrary field. Let $\mathbb{K}$ be a field and let $\bq(x_1,\dots,x_m)$ be a quadratic form with coefficients in $\mathbb{K}$. To this quadratic from, one can assign the following unit sphere 
$${\bf S}^m_\bq:=\left\{(x_1,\dots,x_m)\in\mathbb{K}^m: \bq(x_1,\dots,x_m)=1\right\}.$$
The {\bf quadratic graph} associated to $\bq(x_1,\dots,x_m)$ is, by definition, the Cayley graph $\Cay(\mathbb{K}^m,{\bf S}_\bq^m)$. Note that the unit distance graph is a special case of this construction corresponding to $\mathbb{K}=\RR$ and the standard quadratic form $q(x_1,\dots,x_m)=x_1^2+\dots+x_m^2$. Woodall~\cite{Woodall} studied the quadratic graph $\Cay(\QQ^2,{\bf S}_\bq^2)$ for the form $\bq(x_1,x_2)=x_1^2+x_2^2$ and showed that $\chi(\Cay(\QQ^2,{\bf S}_\bq^2))=2$.
  
The problem of determining the Borel chromatic number with respect to an   unbounded set $S$ is more subtle, and seems to depend on the behavior of $S$ at infinity. As a motivating example, suppose $\K$ is the field of real or $p$-adic numbers, and let $\bq(x_1,\dots,x_m)$ be an arbitrary (non-degenerate) quadratic form with coefficients in $\K$.  The following dichotomy 
was established in ~\cite{Bar-Kei}: either $\bq$ is {anisotropic}, that is, there exists no non-zero vector ${0\neq\bf x}\in\K^m$ with $\bq({\bf x}) = 0$, in which case $S$ is compact and $\bchi(\Cay(\K^m,{\bf S}_\bq^m)) $ is finite, or $\bq$ is isotropic in which case $S$ is unbounded and $\bchi(\Cay(\K^m,{\bf S}_\bq^m) )=\aleph_0$. As an application, it was shown that both in real and $p$-adic case, if $\SL_2(\K)$  is partitioned into finitely many Borel sets then at least one of these sets contains matrices $A, B$ such that $\det(A+B)=0$. 

The problem of computing the Borel chromatic number of $\Cay(\K^m,{\bf S}_\bq^m)$ is closely related to finding bounds for certain oscillatory integrals. It is worth mentioning that in the $p$-adic case it is related to $p$-adic Bessel functions developed by Sally and Taibleson~\cite{Sally}. On the other hand when $\K$ is a finite field, computing the chromatic number of quadratic graphs boils down to estimating Kloosterman sums which can be considered as a finite version of $p$-adic Bessel functions. We refer the reader to~\cites{Bar-Kei, MBKloo} for more details.

These special cases naturally suggest the question of studying the Borel chromatic number for the Cayley graphs of $\RR^m$ and $\QQ_p^m$ with respect to more general classes of algebraically defined sets $S$. In particular, we can ask:

\begin{question}\label{mainq}
Let $\K$ be the field of real or $p$-adic numbers, and assume that $S \subseteq \K^m$ is algebraic, that is, $S$ is the intersection of the zero sets of a finite number of polynomials with coefficients in $\K$. Suppose $S$ is symmetric, unbounded and does not contain zero. When is the Borel chromatic number of $\Cay(\K^m, S)$ infinite?
\end{question}

One of the main results of this paper will partially answer this question. In order to state the results and ideas of the proof we will need to set some notation.

\subsection{Notation}
 
Throughout this paper $\K$ will denote either $\RR$, the field of real numbers, or $\QQ_p$, the field of $p$-adic numbers, for a prime number $p$. The vector space $\K^m$ is equipped with the euclidean norm when $\K=\RR$, and with the supremum norm with respect to the $p$-adic metric when $\K=\QQ_p$. Both norms will be denoted by $\|\cdot\|$. The ball of radius $r$ centered at zero with respect to the metric induced by these norms is denoted by $B_r$.
Moreover, throughout we use the shorthand $\e:=e=2.718\dots$ when $\K=\RR$, and $\e:=p$ when $\K=\QQ_p$.

As a locally compact topological group $\K^m$ has a Haar measure and the measure of a Borel set $A\subseteq \K^m$ will be denoted by $|A|$. When $\K=\QQ_p$, the Haar measure is normalized so that $\ZZ_p$, the ring of $p$-adic integers, carries measure one. For a set $I \subseteq \K^m$, the 
difference set $I-I$ consists of all differences $i_1-i_2$ with $i_1,i_2 \in I$.

Let $0\not\in S\subseteq \K^m$ be an arbitrary Borel set.  An {\bf independent set} of $\G:=\Cay(\K^m,\pm S)$ is a {\bf Borel subset} of $\K^m$ which does not contain a pair of adjacent vertices. The {\bf upper density} of any Borel set $I\subseteq \K^m$ is defined by
 \begin{equation}\label{upper-density-def}
\overline{d}(I):=\limsup_{r\to\infty}\frac{|I\cap B_r|}{|B_r|}.
\end{equation}
The {\bf independence ratio} of $\G$ is defined by 
\begin{equation}
\bar{\alpha}(\G):=\sup\left\{\overline{d}(I): \, \text{$I$ is an independent set of $\G$} \right\}.
\end{equation}

\subsection{Main results} Having set the notation, we are now ready to state our first result.

\begin{theorem}\label{Inde-Borel}
Let $\K$ be as above, and let $f_1,\dots,f_m \in \K[x]$ be $m\geq 2$ polynomials such that $1,f_1,\dots,f_m$ are $\K$-linearly independent. Then there exist constants $C>0$ and $a_0\geq 0$ depending only on $f_1,\dots,f_m$, such that for any $T> a > a_0$ we have
\begin{equation}\label{bound-independence}
0<\bar{\alpha}(\Cay(\K^m,S_{a,T}))\leq C/(T-a),
\end{equation}
where 
$$
S_{a,T}=S_{a,T, f_1,\dots,f_m}:=\left\{\pm \left(f_1(s),\dots,f_m(s)\right)\in \K^m\setminus\{0\}:\, s\in\K, \quad \e^a\leq  \|s\| \leq \e^T \right\},
$$
where $\e=e$ is the base of natural logarithm when $\K=\RR$, and $\e=p$, $T,a\in\mathbb{Z}$ when $\K=\QQ_p$. 
\end{theorem}

\begin{remark}\label{upper-quantitative} 
By a theorem of Steinhaus~\cite{Stromberg}, the difference set of a set of positive measure contains an open neighborhood of zero. From this it follows immediately that if zero
is in the closure of $S_{a,T}$ then $\bar{\alpha}(\Cay(\K^m,S_{a,T}))=0$. Hence in ~\cref{Inde-Borel} we must pick $a_0$ large enough so that the origin is not in the closure of $S_{a,T}$.
\end{remark}

When $\bchi(\G)< \infty$, the Borel chromatic number and the independence ratio of a topological graph $\G$ are related by the inequality
\begin{equation}\label{Ind-Chor}
\bchi(\G)\bar{\alpha}(\G)\geq 1. 
\end{equation}
Since  $S_{a,T}$ is a compact set which does not contain zero, one can easily see that the Borel chromatic number of $\Cay(\K^m,S_{a,T})$ is finite and thus the independent ratio of $\Cay(\K^m,S_{a,T})$ is strictly positive. Hence by combining~\cref{Ind-Chor} and ~\cref{Inde-Borel}, we obtain the following result. 
\begin{corollary}\label{Finite-CayBor} With the same notations as~\cref{Inde-Borel}, we have:
\begin{equation}
\frac{T-a}{C}\leq \bchi(\Cay(\K^m,S_{a,T})).
\end{equation}
\end{corollary}

Let us describe the key ideas of the proof of ~\cref{Inde-Borel}. 
As in \cite{Bar-Kei}, the proof is based on finding 
a spectral upper bound for the independence ratio of 
$\Cay(\K^m,S_{a,T})$ using an analytic analog of the Hoffman bound from~\cite{Bachoc2}. 
The crucial step in the proof relies on establishing a lower bound for a family of oscillatory real and $p$-adic integrals that arise as the Fourier transform of carefully chosen probability measures on $S_{a,T}$. Let us remark that the analysis in \cite{Bar-Kei} also takes advantage of 
a family of probability measures on the hyperbola $xy=1$ constructed as the push-forward of the normalized truncated Haar measure of the multiplicative group $\RR^{\ast}$ via the map $t \mapsto (t, t^{-1})$. However, the homogeneous structure of the hyperbola does not have an analog in the general setting considered in the current work. One novelty of this work is in finding an analogous family of measures on the given polynomial curve {\it in the absence} of the symmetry used in the previous work.
Aside from this, a number of technical obstacles must be overcome in our new setting, as the generality of the measures considered here does not allow the van der Corput lemma to be applied directly to obtain the desired lower bounds for the Fourier transform of these measures. It turns out that one has to handle the frequencies that lie between two specific affine hyperplanes differently from the rest. In each case, one uses a suitable partition of the domain into a uniformly bounded number of intervals (sphere in the $p$-adic case) in order to obtain a required uniform lower bound.

It is worth mentioning that a discrete analog of Theorem \ref{Inde-Borel} has been proven by Furstenberg and S\'{a}rk\"{o}zy~\cite[Proposition 3.19]{Furstenberg2}. According to this theorem, if $f(x)\in \ZZ[x]$ and $f(0) = 0$, then for all sets $I\subseteq\ZZ$ of positive upper density, there exist distinct $x,y \in I$ such that the equation 
$x-y=f(n)$
has a solution for some $n \in \ZZ$. Lyall and Magyar obtained a quantitative version of Furstenberg-S\'{a}rk\"{o}zy's theorem. More precisely, they proved~\cite{Lyall-Magyar} that for $I\subseteq [1,N]\cap \ZZ$ and a family of linearly independent polynomials $f_1(x),\dots, f_m(x)$ in $\ZZ[x]$ with $f_i(0)=0$,  if
$
\{f_1(d),\dots, f_m(d)\}\not\subseteq I-I,
$
for any $0\neq d\in \ZZ$, then  
\begin{equation}\label{k=deg-Lyall}
|I|/N\leq C\left((\log\log N)^2/\log N\right)^{1/m(k-1)},\qquad k:=\max_{1\leq i\leq m}\deg p_i,
\end{equation}
for some absolute constant $C=C(p_1,\dots,p_m)$.  The following reformulation of 
\cref{Inde-Borel} highlights its similarity to Furstenberg-S\'{a}rk\"{o}zy's theorem.
\begin{corollary}[High dimensional polynomial configurations]\label{upper-density-def-Cor}  Assume that $f_1,\dots,f_m$, $a_0$, $C$ and $\e$ are as in~\cref{Inde-Borel}.
For any $T>a>a_0$ and any Borel set $I\subseteq \K^m$ with 
$$\overline{d}(I)>C/(T-a),$$
there exists $s\in \K$ with $\e^a\leq \|s\|\leq \e^T$ and distinct $\x_1,\x_2\in I$ such that
$$\x_1-\x_2=(f_1(s),\dots, f_m(s)).$$
\end{corollary}

\begin{remark}\label{FS-Z} Let us remark that even when $f_1, \dots, f_m$ have integer coefficients, the above theorem is not a formal consequence of the theorem of Furstenberg and S\'{a}rk\"{o}zy. For instance take $I=\{(x,y)\in \RR^2: x^2+y^2\in [0,1/4]+9\mathbb{N}\}$, which has positive density in $\RR^2$, and let $f_1(x)=x$ and $f_2(x)=x^2+1$. Observe that $I\cap\ZZ^2$ has positive upper density in $\ZZ^2$ since $(3\mathbb{N})^2\subseteq I\cap\ZZ^2$. Due to arithmetic obstructions, there are no $\x_1,\x_2\in I\cap \ZZ^2$ and $d\in \ZZ$ such that $\x_1-\x_2=(f_1(d),f_2(d))$.
\end{remark}
\begin{corollary}[Polynomial configurations]\label{real-Sarkozy-quantitive} Let $f_1,\dots,f_m$, $a_0$, $C$ and $\e$ be as in~\cref{Inde-Borel}.
For any $T>a>a_0$ and any Borel set $I\subseteq \K$ with 
$$\overline{d}(I)>\left(\frac{C}{T-a}\right)^{1/m},$$
there exists $s\in \K$ with $\e^a\leq \|s\|\leq \e^T$ such that
$$\{f_1(s),\dots, f_m(s)\}\subseteq I-I.$$
\end{corollary}
\begin{proof}
Let $I$ be as above. Then $\overline{d}(I^m)>C/(T-a)$ and so from~\cref{upper-density-def-Cor} we obtain the result.  
\end{proof}

Let us now turn to studying the Borel chromatic number of Cayley graphs assigned to algebraic varieties.  
\begin{remark}\label{Steinhaus1}
By Steinhaus theorem, mentioned in~\cref{upper-quantitative}, if zero is in the topological closure of $S\subseteq \K^m$ then $\bchi(\Cay(\K^m,\pm S))>\aleph_0$ (see~\cref{closuer-chrom}). Conversely, if $0 \not\in \overline{S}$, the closure of $S$, then the Borel chromatic number 
of $\Cay(\K^m,\pm S)$ is at most $\aleph_0$. Moreover, one can easily show that $\bchi(\Cay(\K^m,\pm S))$ is finite when $S$ is bounded and $0\not\in\overline{S}$. 
\end{remark}
From ~\cref{Finite-CayBor} and ~\cref{Steinhaus1} we deduce the following qualitative result, which gives an answer to a special case of \cref{mainq}:
\begin{corollary}[Infinite Borel chromatic number]\label{chrom-Borel}
Let $f_1,\dots,f_m$ be as in~\cref{Inde-Borel}. Then for all sufficiently large $a$ we have
\begin{equation}
\bchi(\Cay(\K^m, S_a)) = \aleph_0,
\end{equation}
where 
$$
S_a=\left\{\pm (f_1(s),\dots,f_m(s))\in \K^m:\, s\in\K, \quad a\leq  \|s\| \right\}.
$$
\end{corollary}
The assumption that $1,f_1,\dots, f_m$ are linearly independent is quite natural in this context. In fact if 
$1,f_1,\dots, f_m$ are linearly dependent, then $S_a$ will be contained in an affine hyperplane in $\K^m$ and one can easily see that the Borel chromatic number of the Cayley graph associated to a hyperplane which does not pass through the origin is finite. More generally, we have

\begin{theorem}[Multivariate polynomial configurations]\label{dim}
Let  $F_1,\dots, F_m\in \K[x_1,\dots,x_d]$ be $m\geq 2$ polynomials such that $1,F_1,\dots,F_m$ are linearly independent. Then for any $\delta>0$ we have
\begin{equation}
\bchi(\Cay(\K^m,S_\delta))=\aleph_0,
\end{equation}
where 
$
S_\delta=\{\pm (F_1({\bf s}),\dots,F_m({\bf s}))\in\K^m:\, {\bf s}\in\K^d\}\setminus B_\delta.
$
\end{theorem}

\begin{proof}
Suppose that the maximum degree of $F_1, \dots, F_m$ is $\ell-1$. Set $n_i= \ell^{i-1}$ for all $ 1 \le i \le d$ and substitute $x_i$ with $t^{n_i}$. Thus the monomial $x_1^{ \alpha_1} \cdots x_d^{ \alpha_d}$ will be substituted by $t^{h( \alpha_1, \dots, \alpha_d)}$, where
$$h( \alpha_1, \dots, \alpha_d)= \sum_{i=1}^{d} \alpha_i \ell^{i-1}.$$
 Notice that $h$ is injective on the
cube defined by $ 0 \le \alpha_i \le \ell-1$. This clearly follows from the fact that $\alpha_1, \dots, \alpha_d$ are
digits of $h( \alpha_1, \dots, \alpha_d)$ when expressed in base $\ell$, and hence are uniquely determined by $h( \alpha_1, \dots, \alpha_d)$. Thus we obtain $m\geq 2$ polynomials 
$f_i(t):= F_i(t^{n_1},\dots,t^{n_d}),$
 such that $1, f_1,\dots, f_m$ are $\K$-linearly independent.  Then from~\cref{Steinhaus1} and~\cref{chrom-Borel} we obtain the result. 
\end{proof}

Let us now address a number of questions that naturally arise in connection with theorems stated above. 
Recall that the {\bf clique number} of a graph $\G$, denoted by $\omega(\G)$, is the largest $n$ for which $\G$ has a subgraph isomorphic to the complete graph $K_n$. Since $\omega(\G)\leq \bchi(\G)$, the claim in ~\cref{chrom-Borel}, and hence~\cref{dim}, would be trivial if one could exhibit arbitrarily large cliques in $\G$. However, we will show that there is an algebraic obstruction to this, suggesting that the infinitude of the Borel chromatic number cannot be proven by such local arguments. Obviously, it is enough to prove the nonexistence of large cliques in these graphs when the base field is replaced by an algebraic closure. 
\begin{theorem}\label{clique} Let $\KK$ be an algebraic closed field of characteristic $0$, and let $V\subseteq  \KK^m$ be an irreducible  variety with $\dim V=1$ that is not an affine line. Then
\begin{equation}
\omega(\Cay(\KK^m, \pm V))< \infty.
\end{equation}
\end{theorem}

Another point that has to be addressed is the extent to which the polynomial nature of the map is essential for the above results to hold. The next theorem shows that they cannot be replaced by real analytic functions.

\begin{theorem}\label{analytic}
 There exists a real analytic curve $f: (0,1) \to \RR^2$ such that the image of $f$ does not lie between any 
two parallel lines and the Borel chromatic number of the Cayley graph of $\RR^2$ with respect to its graph 
$$\{ \pm (s, f(s)): 0<s<1 \},$$
 is finite. 
\end{theorem}

As it was earlier mentioned, proving statements analogous to~\cref{dim} for the ordinary chromatic number
seems to be quite difficult. This difficulty can be traced back to the fact that proving such statements is equivalent to finding (or proving the existence of)  finite subgraphs of arbitrarily large chromatic number inside the Cayley graphs. In the last section of this paper, we will prove a general theorem that highlights this point by showing that the question of determining the chromatic number over $\CC$ is indeed algebraic in nature. 

Let $V$ be an algebraic set defined over $\QQ$, given as the intersection of the zero set of polynomials $F_i\in \QQ[x_1,\dots,x_m]$ for $1\leq i\leq n$. Without loss of generality, we can assume that $F_i\in \ZZ[x_1,\dots,x_m]$. If $R$ is any ring of characteristic zero, we can define
the $R$-points of $V$ as 
\[ V(R)=\left\{(x_1, \dots, x_m) \in R^m: F_i(x_1, \dots, x_m)=0, \quad 1 \le i \le n \right\}.\]
For $p$ prime, we will denote by $\ZZ_p$ the ring of $p$-adic integers. 

\begin{theorem}\label{Cass-imb} Let $V$ be an irreducible variety defined over $\QQ$. Then
\begin{equation}
 \chi(\Cay(\CC^m, \pm V(\CC)))=\sup_{p}\chi\left(\Cay(\ZZ_p^m, \pm V(\ZZ_p))\right),
\end{equation}
if one of the two sides is finite. The sup is taken over all primes $p$.
\end{theorem} 
The proof of  ~\cref{Cass-imb} relies on a theorem of de Bruijn and Erd\H{o}s which relates the chromatic number of a graph to its finite subgraphs and an embedding theorem of Cassels~\cite{Cassels}. 
\section{A spectral bound for the independence ratio}
Let $\G=(V,E)$ be a finite regular graph with $n$ vertices and the adjacency matrix $A$. Let 
$$\lambda_{\max}:=\lambda_0\geq \lambda_1\geq \cdots\geq \lambda_{n-1}=:\lambda_{\min},$$ 
denote the spectrum of $A$, and assume that $I\subseteq V$ is an independent set of $\G$. The celebrated Hoffman bound states that
\begin{equation}\label{inde-Hoff}
\frac{|I|}{n}\leq -\frac{\lambda_{n-1}}{\lambda_0-\lambda_{n-1}}=-\frac{\lambda_{\mathrm{min}}}{\lambda_{\mathrm{max}}-\lambda_{\mathrm{min}}}. 
\end{equation}
It is a reinterpretation of this inequality that is the key to the generalization we will need later.
Denote by $L^2(V)$ the Hilbert space of complex-valued functions on $V$, equipped with the inner product 
$$\langle f,g\rangle:=\sum_{v\in V}f(v)\overline{g(v)}.$$
The adjacency operator $A: L^2(V)\to L^2(V)$, defined by $Af(v)=\sum_{vw\in E} f(w)$, is easily seen to be self-adjoint. Further, one can see that its {\bf numerical range} defined by
$$
W(A):=\left\{\langle Af,f\rangle: \|f\|_2=1\right\},
$$  
is equal to $[\lambda_{n-1},\lambda_0]$. It is a routine verification that a subset $I\subseteq V$ is independent in $\G$ if and only if for all $f\in L^2(V)$ supported on $I$, one has $\langle Af,f\rangle=0$.
This novel interpretation of independent sets is used in~\cite{Bachoc2} to prove an analog of the Hoffman bound for certain Cayley graphs of the Euclidean additive group $\RR^n$. As we will need to work in a slightly more general framework, it will be useful to briefly review some of the key points of~\cite{Bachoc2}; for details, we refer the reader to the original paper~\cite{Bachoc2}. 

Abstractly, let $(V, \Sigma, \nu)$ be a {probability space}, consisting of a set $V$, a $\sigma$-algebra $\Sigma$ on $V$, and a probability measure $\nu$, and consider the Hilbert space $L^2(V)$ of square integrable functions with respect to the inner product:
\begin{equation*}
\langle f,g\rangle=\int_V f(x)\overline{g(x)}\,\d{\nu(x)}.
\end{equation*}
For a bounded and self-adjoint operator $A: L^2(V)\to L^2(V)$, one can show that the {\bf numerical range} of $A$, defined by 
$$
W(A)=\left\{\langle Af,f\rangle: \|f\|_2=1\right\},
$$
is an interval in $\RR$. We denote the endpoints of $W(A)$ by 
$$m(A):=\inf\left\{\langle Af, f\rangle: \|f\|_2 = 1 \right\}, \qquad M(A):=\sup\left\{\langle Af, f\rangle: \|f\|_2=1\right\}.$$

\begin{definition} Let $A: L^2(V)\to L^2(V)$ be a bounded, self-adjoint operator. A
measurable set $I\subseteq V$ is called an {\bf independent set} for $A$ if $\langle Af, f\rangle=0$ for each $f\in L^2(V)$ which vanishes almost everywhere outside of $I$. Moreover,  the {\bf chromatic number} of $A$, denoted by $\chi(A)$, equals the least number $k$ such that one can partition $V$ into $k$ independent sets for $A$.
\end{definition}
The {\bf independence ratio} of $A$ is defined by
\begin{equation}
\bar{\alpha}(A):=\sup\left\{\nu(I):\, \text{$I$ is an independent set of $A$ }\right\}. 
\end{equation}
The following theorem, which can be obtained by a clever modification of the proof of Hoffman's
bound presented by Bollob\'{a}s~\cite[Chapter VIII.2]{Bollobas}, recovers~\cref{inde-Hoff} when $A$ is the adjacency matrix of a finite regular graph. 
\begin{theorem}\label{Inde-Operator-Bound} Let $(V,\Sigma,\nu)$ be a probability space and let $A: L^2(V)\to L^2(V)$
be a nonzero, bounded, self-adjoint operator. Fix a real number $R$ and set $\varepsilon=\|A\1-R\1\|_2$, where $\1=\1_V$ is the characteristic function of $V$. Suppose there exists a set $I\subseteq V$ with $\nu(I) > 0$ which is
independent for $A$. Then, if $R-m(A)-\epsilon> 0$, we have
\begin{equation}
\bar{\alpha}(A)\leq\frac{-m(A)+2\epsilon}{R-m(A)-\epsilon}.
\end{equation}
\end{theorem}
\begin{proof}
See~\cite[Theorem 2.2]{Bachoc2}.
\end{proof}

\begin{remark}\label{positive-negative-0} Let $A: L^2(V)\to L^2(V)$ be a nonzero, bounded and self-adjoint operator. Moreover assume that $\chi(A)<\infty$. Bachoc, DeCorte, de Oliveira and Vallentin proved~\cite[Theorem 2.3]{Bachoc2}
\begin{equation}\label{Hof-oper}
\chi(A)\geq 1-\frac{M(A)}{m(A)},
\end{equation}
which is an analogue of Hoffman's bound for the chromatic number.
Similar to the Hoffman bound, when $A\neq 0$ and $\chi(A)<\infty$ the proof of~\cref{Hof-oper} shows that $m(A)<0$ and $M(A)>0$.
\end{remark}

We now apply these theorems to estimate the independence ratio of Cayley graphs of $\K^m$, where, as before, $\K=\RR$ or $\K=\QQ_p$ for a prime $p$. All that follows is parallel to~\cite{Bachoc2}, where the case $\K=\RR$ is dealt with; in fact, arguments in~\cite{Bachoc2} can be easily seen to work also in the $p$-adic case. For the convenience of the reader we will briefly sketch the key points.

Denote by $\d{x}$ a Haar measure on $\K$. When $\K=\QQ_p$, we assume that it is normalized such that $\int_{\ZZ_p} \d{x}=1$. We will also write $\d{x}$ for the product (Haar) measure $\d{x_1} \cdots \d{x_m}$ on $\K^m$. For a Borel set $E\subseteq \K^m$ we use $|E|=\int_E \d{x}$ to denote the measure of $E$. 
Throughout we will assume that  $S\subseteq \K^m$ is a Borel set which is bounded  and symmetric (i.e., $-S=S$) and does not contain the origin in its closure.  
Let $\mu$ be a {\bf Borel probability measure} on $\K^m$ supported in $S$. We will also assume that $\mu$ is symmetric, i.e., $\mu(-E) =\mu(E)$ holds for all Borel sets $E$. One can easily verify that the following operator is  a bounded and self-adjoint operator~\cite[Proposition 8.49]{Folland} 
$$
A_\mu: L^2(\K^m)\to L^2(\K^m),\qquad f\mapsto f*\mu,$$
where  
\begin{equation}
f*\mu(x)=\int_{\K^m} f(x-y)\,d\mu(y).
\end{equation}
The numerical range of $A_\mu$ can now be determined using Fourier analysis. Below, we will review a number of basic properties of the Fourier transform over $\RR$ and $\QQ_p$. For details we refer the reader to~\cite{Rudin} for harmonic analysis over $\RR$ and~\cite{Taibleson} for $p$-adic harmonic analysis.
When $\K=\RR$, we will use the character  
$$\psi: \RR \to \CC^{\ast}, \quad \psi(x)=\exp(2\pi i x).$$ When $\K=\QQ_p$, for $x\in\K$, we write $n_x$ for the smallest non-negative integer such that $p^{n_x}x\in\ZZ_p$. Let $r_x$ be an integer such that $r_x\equiv p^{n_x}x \pmod{p^{n_x}}$. It is well-known that the following map (called the Tate character)
\begin{equation}\label{additive-char}
\psi: \QQ_p\to \CC^*,\qquad \psi(x)=\exp \left( {\frac{2\pi i r_x}{p^{n_x}} } \right),
\end{equation}
is a non-trivial character of $(\QQ_p,+)$ with the kernel $\ZZ_p$. 

The Fourier transforms of $\mu$ and $f \in L^1(\K^m)$ are respectively defined by the integrals
$$
\widehat{f}(u)=\int_{\K^m}\overline{\psi}(x\cdot u)f(x)\, \d{x},\qquad \widehat{\mu}(u)=\int_{\K^m} \overline{\psi}( x\cdot u)\, \d{\mu(x)}.
$$
Here $x\cdot u$ is the standard bilinear form on $\K^m$, and $ \overline{\psi }$ is the complex conjugate of $\psi$. We remark that since
$\mu$ is symmetric, its Fourier transform $\widehat{\mu}$ is a real-valued function.
By Plancherel's theorem the Fourier transform extends to an isometry on $L^2(\K^m)$ and thus for any $f\in L^2(\K^m)$ we have
\begin{equation}\label{mult-op}
\langle A_\mu f,f \rangle=\langle \widehat{A_\mu f},\widehat{f} \rangle=\langle \widehat{f*\mu},\widehat{f} \rangle=\langle \widehat{\mu}\widehat{f},\widehat{f}\rangle. 
\end{equation}

\begin{lemma}\label{m=inf}
Let $S \subseteq \K^m$ be a bounded symmetric Borel set which does not contain the origin in its closure, and 
let $\mu$ be a symmetric  Borel probability measure on $\K^m$ supported  on $S$. Then the numerical range of $A_\mu$ is given by
$$
m(A_\mu)=\inf_{u\in \K^m}\widehat{\mu}(u),\qquad 1=M(A_\mu)=\sup_{u\in \K^m}\widehat{\mu}(u).
$$
\end{lemma}
\begin{proof}
From \cref{mult-op} combined with the fact that the Fourier transform on $L^2(\K^m)$ is an isometry, we   
deduce that the numerical range of the operator $A_\mu$ is the same as the numerical range of the multiplication operator $g\mapsto \widehat{\mu} g$. Since $\mu$ is a symmetric  Borel probability measure, $\widehat{\mu}$ is a bounded continuous  real-valued function. Now let $g\in L^2(\K^m)$ with $\|g\|_2=1$. Evidently 
$$
\langle \widehat{\mu} g,g\rangle=\int_{\K^m} \widehat{\mu}(x)|g(x)|^2\, \d{x}\geq \inf_{u\in \K^m}\widehat{\mu}(u),
$$
and so $m(A_\mu)\geq \inf_{u\in \K^m}\widehat{\mu}(u).$
Now let $\varepsilon>0$ and pick $x_0\in \K^m$ with $\widehat{\mu}(x_0)\le\inf_{u\in \K^m}\widehat{\mu}(u)+\varepsilon$. Let $B_{ \delta}=B_{\delta}(x_0)$ be the ball of radius $ \delta$ centered at $x_0$. Since $\widehat{\mu}$ is continuous, we conclude that 
$$
\lim_{ \delta \to 0}\int_{\K^m}\widehat{\mu}(x)\left|
\frac{\1_{B_{ \delta}}(x)}{\sqrt{|B_{ \delta}|}}\right|^2\, \d{x}=\lim_{ \delta \to 0} \frac{1}{|B_{ \delta}|} \int_{B_{ \delta}}\widehat{\mu}(x)\, \d{x}=\widehat{\mu}(x_0),
$$  
where $\1_{B_{ \delta}}$ is the characteristic function of $B_{ \delta}$. Hence 
$m(A_\mu)\leq \widehat{\mu}(x_0)\leq \inf_{u\in \K^m}\widehat{\mu}(u)+\varepsilon,$ and so we have $m(A_\mu)=\inf_{u\in \K^m}\widehat{\mu}(u)$. Similarly, since $\mu$ is a probability measure, we obtain $M(A_\mu)=\sup_{u\in \K^m}\widehat{\mu}(u)=1$. 
\end{proof}
Let us now consider the independence ratio of $\Cay(\K^m,\SC)$. First note that 
 $\mathrm{Supp}(\mu)\subseteq S$ and so any independent set $I$ of $\Cay(\K^m,S)$  is also an independent set of $A_\mu$. Therefore
\begin{equation}\label{chiGraph-chiOpe}
\bar{\alpha}(\Cay(\K^m,S))\leq \bar{\alpha}(A_\mu).
\end{equation}
Since $S$ is bounded and the origin is not in the closure of $S$, it is easy to see that the Borel chromatic number of $\Cay(\K^m,S)$ is finite. Therefore from~\cref{positive-negative-0} we have
\begin{equation}\label{positive-negative}
m(A_\mu)<0.
\end{equation}

 For $r > 0$, let $B_r\subseteq \K^m$ be the ball of radius $r$ centered at the origin, and normalize the induced Haar measure on $B_r$. 
Define 
$A_\mu^r: L^2(B_r) \to L^2(B_r),$
 by
$$
A_\mu^r(f):=\left(A_\mu \bar{f}\,\right){|_{B_r}},
$$
where $\bar{f}$ is the extension of $f$ to $\K^m$ defined to be zero on $\K^m\setminus B_r$. By abuse of notation, we will continue to write $f$ for $\bar{f}$.  
Now let $I$ be an independent set of $A_\mu$. Hence $I\cap B_r$ is an independent set for the operator $A_\mu^r$. 
From the definition of the independence ratio of an operator, we obtain the following density bound:
$$
\frac{|B_r\cap I|}{|B_r|}\leq \bar{\alpha}(A_\mu^r).
$$
This implies that $\bar{\alpha}(A_\mu)\leq \limsup_{r\to \infty}\bar{\alpha}(A_\mu^r)$ and so from~\eqref{chiGraph-chiOpe} we deduce 
\begin{equation}\label{sup}
\bar{\alpha}(\Cay(\K^m,S))\leq \limsup_{r\to \infty}\bar{\alpha}(A_\mu^r).
\end{equation}
For a given $r > 0$, define
\begin{equation}
R = R(r) = \langle A_\mu^r\1_{B_r}, \1_{B_r}\rangle_{L^2(B_r)},\quad \epsilon=\epsilon(r) =\|A_\mu^r\1_{B_r}-R(r)\1_{B_r}\|_{L^2(B_r)},
\end{equation}
where $\1_{B_r}$ is the characteristic function of $B_r$.
\begin{lemma}\label{Re} With the above notation, we have 
\begin{equation}\label{Re-eq}
\lim_{r\to \infty}m(A_\mu^r)=m(A_\mu),\quad \lim_{r\to \infty} R(r)=1,\quad \lim_{r\to \infty} \epsilon(r)=0.
\end{equation}
\end{lemma}
Before proving this lemma, we recall that the $p$-adic norm on $\QQ_p$ has the ultrametric property, i.e., $\|x+y\|\leq \max\{\|x\|,\|y\|\},$ with equality if $\|x\|\neq \|y\|$.
\begin{proof}[Proof of~\cref{Re}]
Let $f\in L^2(\K^m)$. Define $f^r:=f_{|_{B_r}}$. Then 
$$
\langle A_\mu^r f^r, f^r\rangle_{L^2(B_r)}=\frac{\langle A_\mu f^r, f^r\rangle_{L^2(\K^m)}}{|B_r|}\geq m(A_\mu)\frac{\langle f^r, f^r\rangle_{L^2(\K^m)}}{|B_r|}=m(A_\mu)\langle f^r, f^r\rangle_{L^2(B_r)}. 
$$
Hence $m(A_\mu^r)\geq m(A_\mu)$. Conversely, as $ r \to \infty$, we have
$$
m(A_\mu^r)\leq \frac{\langle A_\mu^r f^r, f^r\rangle_{L^2(B_r)}}{\langle f^r, f^r\rangle_{L^2(B_r)}}=\frac{\langle A_\mu f^r, f^r\rangle_{L^2(\K^m)}}{\langle f^r, f^r\rangle_{L^2(\K^m)}}\longrightarrow \frac{\langle A_\mu f, f \rangle_{L^2(\K^m)}}{\langle f, f\rangle_{L^2(\K^m)}},
$$
since $f^r\to f$ in $L^2(\K^m)$. This shows that $\lim_{r\to\infty}m(A_\mu^r)=m(A_\mu)$. 

Recall that $\mu$ is a probability measure with $\mathrm{Supp}(\mu)\subseteq \SC$. Thus for every $x\in B_r$ we have
\begin{equation}\label{A_mu}
A_\mu^r\1_{B_r}(x)=\int_{\SC} \1_{B_r}(x-y)\, \d{\mu(y)}.
\end{equation}
Since $\SC$ is bounded, we can assume that $\SC\subseteq B_d$ for some $d>0$. When $\K=\QQ_p$,  then for  $r>d$, by the ultrametric inequality, we obtain $A_\mu^r\1_{B_r}(x)=1$ from which~\eqref{Re-eq} follows.

 Now assume that $\K=\RR$. Let $r>d$ and $x\in B_{r-d}$. Then from~\cref{A_mu} we deduce that $A_\mu^r\1_{B_r}(x)=1$ for all $x\in B_{r-d}$. Hence 
$
\lim_{r\to\infty} R(r)=\lim_{r\to\infty}\frac{1}{|B_r|}\int_{B_r} A_\mu^r\1_{B_r}(x)\, \d{x}=1. 
$
Finally
$$
\frac{1}{|B_r|}\int_{B_r} |A_\mu^r\1_{B_r}(x)-1|^2\, \d{x}\leq 2\frac{|B_r\setminus B_{r-d}|}{|B_r|}\to 0, \quad r\to\infty.
$$
Then by the triangular inequality we deduce that $\lim_{r\to\infty} \epsilon(r)=0$. 
\end{proof}
Combining~\cref{Inde-Operator-Bound}, ~\cref{m=inf,Re} and~\cref{positive-negative,sup} we obtain the following theorem:
\begin{theorem}\label{indep-Hoff-spectral} Let $S$ be a bounded, symmetric Borel subset of $\K^m$ which does not contain the origin in its closure. Then for any symmetric, Borel probability measure $\mu$ on $\K^m$ with support contained in $S$ we have
\begin{equation}
\bar{\alpha}(\Cay(\K^m,S))\leq -\frac{\inf_{u\in\K^m}\widehat{\mu}(u)}{1-\inf_{u\in\K^m}\widehat{\mu}(u)}.
\end{equation}
\end{theorem}

We conclude this section with the following general fact.
\begin{lemma}\label{closuer-chrom} Let $S\subset\K^m$ be a symmetric Borel set, and assume that $0\not\in S$. Then the Borel chromatic number of 
$\Cay(\K^m,S)$ is at most $\aleph_0$,
if and only if the origin is not in the closure of $S$. 
\end{lemma}
\begin{proof} First assume that $0\not\in \overline{S}$. Choose $\delta>0$ such that $B_{2\delta}\cap S=\emptyset$. This, in particular, implies that $B_{\delta}$ is an independent set for $\Cay(\K^m,S)$. Note that since $\QQ^m$ is dense in $\K^m$, we have 
$$\K^m=\bigcup_{v \in \QQ^n} (v+B_{\delta}).$$
 This countable cover can then be easily transformed further into a {\bf disjoint} countable cover of independent Borel (in fact, locally closed) sets, yielding the desirable coloring.    

Now assume  $\bchi(\Cay(\K^m,S))\leq \aleph_0$. From countable additivity, it follows that there exists a color class $I$ with  positive Lebesgue measure. Using the  aforementioned theorem of Steinhaus, there exists $\delta_1>0$ such that $B_{\delta_1}\subseteq (I-I)$. Since $I$ is independent, it follows that $S\cap B_{\delta_1}=\emptyset$, and so $0\not\in \overline{S}$.
\end{proof}
\section{Oscillatory integrals with real polynomial phase} We are now ready to prove ~\cref{Inde-Borel} in the case $\K=\RR$. Throughout this section we fix $m\geq 2$ polynomials $f_1(x)$, $\dots$, $f_m(x)$ with real coefficients such that
\begin{equation}\label{Linearly independent}
1,f_1(x),\dots, f_m(x),
\end{equation}
are linearly independent over $\RR$. Let $n:=\max_{i}\deg f_i$. From~\cref{Linearly independent} we deduce that $m\leq n$. Set
$$a_0= \max \left\{ t\geq 0: f_1(e^t)= \cdots = f_m(e^t)=0 \right\},$$
where $e$ is the Euler number.
Fix $a>a_0$, and for $T>a$, consider the following symmetric, bounded set
$$
S:=S_{a,T}=S_{a,T,f_1,\dots,f_m}=\left\{\pm (f_1(s),\dots, f_m(s))\in \RR^m \setminus \{ 0 \}: s\in [e^a,e^T]\, \right\}.
$$

Since $a>a_0$, we have $0\not\in \overline{S}$. Let $\mu_T$ be the measure on $\RR^m$ defined for every Borel subset $E \subseteq \RR^m$ via
\[ \mu_T(E):= \frac{1}{2(T-a)} \int_{e^a}^{e^T} \frac{\1_E(f_1(s),\dots,f_m(s))+ \1_{(-E)}(f_1(s),\dots,f_m(s))}{s} \, \d{s}. \]
It is easy to verify that $\mu_T$ is a symmetric Borel probability measure on $\RR^m$ and $\mathrm{Supp}(\mu_T) =S$. 
The Fourier transform of this measure is given by
\[
\widehat{\mu_T}(\lambda_1,\dots,\lambda_m)= \frac{1}{T-a} \int_{e^a}^{e^T} \frac{\cos(2 \pi(\lambda_1 f_1(s)+\cdots+\lambda_m f_m(s)))}{s}\, \d{s},\qquad (\lambda_1,\dots,\lambda_m)\in \RR^m; \]
These integrals resemble certain singular integrals that were studied by Stein and Wainger~\cite{Stein-Wainger}. Let $f(x)$ be a polynomial with real coefficients of degree at most $d$, and define
$$
I(f):=p.v. \int_\RR e^{2\pi i f(x)}\frac{\d{x}}{x}.
$$ 
Stein and Wainger~\cite{Stein-Wainger} proved that that $I(f)\leq c_d$, where the constant $c_d$ depends only on the degree of $f$. In our case we need to consider the truncated integral instead of its principal value. For our purposes, it would be convenient to apply the logarithmic change of variables $s=e^t$: 
\[
\widehat{\mu_T}(\lambda_1,\dots,\lambda_m)= \frac{1}{T-a} \int_{a}^{T} \cos\left( 2\pi \left(\lambda_1 f_1(e^t)+\cdots+\lambda_m f_m(e^t)\right)\right)\,  \d{t}. \]
\begin{theorem}\label{lower bound}
There exists a constant $C>0$, depending only on the polynomials ${f_1,\dots,f_m}$, such that for all $ (\lambda_1,\dots,\lambda_m)\in \RR^m$, and any $T > a$ we have
\begin{equation}\label{inequlity-main}
-C/(T-a)\leq \widehat{\mu_T}(\lambda_1,\dots,\lambda_m).
\end{equation}
\end{theorem}
 After completing this paper, we were informed by James Wright that Nagel and Wainger have studied similar oscillatory integrals in connection to $L^2$ boundedness of Hilbert transform. In fact, using several reductions~\cref{lower bound} may be obtained from~\cite[Theorem 3.1 and Corollary 3.6]{Nagel}. Since these steps are also rather lengthy, and also for convenience of readers less familiar with harmonic analysis, we will provide our own elementary proof here, which will also facilitate understanding the proof  in the $p$-adic case.

We will start by proving two simple combinatorial lemmas. In what follows, two intervals $I$ and $I'$ are called disjoint if $I \cap I'$ has no interior point. 

\begin{lemma}\label{interval}
Let $A_1, \dots, A_n \subseteq \RR$ be such that each $A_i$ is a union of at most $n$ intervals. Then there exist $ k\leq 2n^4$ disjoint intervals $I_1, \dots, I_k$ such that 
$\cup_{i=1}^n A_i=\cup_{i=1}^k I_i,$
and for each $1 \le j \le k$, there exists $1 \le i \le n$ such that $I_j \subseteq A_i$.
\end{lemma}

\begin{proof}
This follows easily from the fact that the intersection of two intervals in $\RR$ is either empty or an interval. 
\end{proof}
\begin{lemma}\label{sublevel}
Let $g(t)$ be a polynomial of degree $n\geq 1$ and set $\Phi(t)=g(e^t)$. Then for any $M>0$, the set 
$\{t \in \RR: |\Phi(t)|\geq M\}$ can be split to a disjoint union of intervals $I_1, \dots, I_k$ such that $k \le 3n$ and $\Phi'(t)$ is monotone on each $I_i$.
\end{lemma}
\begin{proof}
Consider the superlevel set $T_M:=\{t \in \RR: |\Phi(t)|> M\}$. As $T_M$ is open, it is either $\RR$ or a union of open intervals 
$I_{ \alpha}, 1 \le \alpha \le \ell$  with each end-point a solution for $\Phi(t)= \pm M$. Note that since the exponential map is strictly increasing, the equation $\Phi(t)=\pm M$ has at most $2n$ solutions, hence, $\ell \le 2n$. Suppose $[a,b]$ is one of these intervals such that the equation $\Phi''(t)=0$ has $s \ge 1$ roots $a\le r_1< \dots < r_s \le b$ in $[a,b]$. Upon replacing each such interval $[a,b]$ by the disjoint union $[a,r_1] \cup (r_1, r_2] \cup \cdots \cup (r_s, b]$ the number of
intervals increases by at most $n$ (a bound for the number of roots of $\Phi''(t)=0$), while $\Phi'$ is monotone over each one of the new intervals. The total number of produced intervals is $n + \ell \le 3n$. 
\end{proof}
Let $(\lambda_1,\dots,\lambda_m)\in\RR^m$ be an arbitrary vector. For the rest of this section we set
\begin{equation}\label{Phi}
\Phi(t):=\Phi_{(\lambda_1,\dots,\lambda_m)}(t)=\lambda_1 f_1(e^t)+\cdots+\lambda_m f_m(e^t).
\end{equation}
We will use the following version of the van der Corput lemma~\cite[Proposition 2.6.7.]{Grafakos}. 
\begin{theorem}[van der Corput's Lemma] Suppose that a real-valued function  $\psi: (a,b)\to \RR$  satisfies 
$$|\psi^{(k)}(t)|\geq \eta>0,$$
 for all $t\in (a,b)$, where $k \ge 1$ is an integer. Then
$$
\left|\int_a^b e^{i\psi(t)}\, \d{t} \right|\leq \frac{12 k}{\eta^{1/k}},
$$
provided, in addition when $k = 1$, that $\psi'(t)$ is monotonic on $(a,b)$.
\end{theorem}

\begin{proof}[Proof of ~\cref{lower bound}]
Since $\widehat{\mu_T}(0,\dots,0)=1$, we can assume that $(\lambda_1,\dots,\lambda_m)$ is not-zero. We will  prove ~\cref{lower bound} by considering two cases.
\subsection{Low frequency case}  Let $(\lambda_1,\dots,\lambda_m)\in\RR^m$ be an arbitrary vector such that 
\begin{equation}
\left|\sum_{i=1}^m \lambda_i f_i(0)\right|\leq 1/8.
\end{equation}
\begin{lemma}\label{Phi^k}
There exist real numbers $\alpha_1, \dots, \alpha_n$, where $n=\max_i \deg f_i$, such that
\begin{equation}\label{eq 1}
\Phi(t) =\sum_{i=1}^m \lambda_i f_i(0)+\sum_{k=1}^{n} \alpha_k \Phi^{(k)}(t).
\end{equation}
\end{lemma}
\begin{proof}
For any $1\leq i\leq m$, let $f_i(t)=\sum_{j=0}^{n} a_{ij} t^j$. Using Lagrange's interpolation method (or Vendermonde determinant) we can find real numbers $\alpha_1,\dots,\alpha_n$, not all zero, such that $\sum_{k=1}^{n} \alpha_k j^k=1$ for every $1\leq j\leq n$.  
 Note that 
$
\Phi(t)= \sum_{i=1}^m\lambda_i f_i(0)+\sum_{i=1}^m\sum_{j=1}^{n}\lambda_i a_{ij} e^{jt}.
$
Hence we obtain
\begin{equation*}
\sum_{k=1}^{n}\alpha_k\Phi^{(k)}(t)=\sum_{k=1}^{n}\sum_{i=1}^m\sum_{j=1}^{n} \alpha_k j^k \lambda_i a_{ij}e^{jt}=\sum_{i=1}^m\sum_{j=1}^{n} \left(\sum_{k=1}^{n} \alpha_k j^k\right) \lambda_i a_{ij}e^{jt}=\Phi(t)-\sum_{i=1}^m\lambda_i f_i(0),
\end{equation*}
which finishes the proof.
\end{proof}
Set  
$$J:=\{ a \le t \le T: |\Phi(t)|< 1/4 \},$$
and let  $\alpha_1, \dots, \alpha_n$ be provided by ~\cref{Phi^k}. Define $H:=\max_{1\leq i\leq n}|\alpha_i|$. 
In the rest of this proof, certain constants will appear whose values will depend only on $n$. We will denote these constants by $\kappa_1(n), \kappa_2(n), \cdots$. The exact values of these constants (which are all of the form $O(n^{O(1)} )$) are not important, and hence we will not keep track of them consistently. 
First, we will show that $[a,T]\setminus J$ can be expressed as a disjoint union of $\kappa_1(n) \le 3n$ intervals such that in each interval $\Phi'(t)$ is monotone and $|\Phi^{(k)}(t)|\geq 1/(8Hn)$ for a {\bf fixed} $1\leq k\leq n$.  

We recall that $1, f_1,\dots,f_m$ are linearly independent and $(\lambda_1,\dots,\lambda_m)$ is not-zero. Thus $\sum_{i=1}^m\lambda_i f_i(t)$ is a polynomial of degree at least $1$ and at most $n$. Hence from ~\cref{sublevel} we know that
$[a,T]\setminus J$ can be written as a disjoint union of at most $3n$ intervals $I'_1, \dots, I'_p$ such that $\Phi'(t)$ is monotone on each $I'_i$.
Moreover notice that for each $1\leq i\leq p$ we have
\begin{equation}\label{partition}
 I'_i\subseteq \bigcup_{k=1}^n \left\{ t\in [a,T]: |\Phi^{(k)}(t)|\geq 1/(8Hn)\right\}.
\end{equation}
In fact, if for some $t\in [a,T]\setminus J$ we have $|\Phi^{(k)}(t)|<\frac{1}{8Hn}$ for all $ 1 \le k \le n$, then using~\cref{eq 1} we obtain
\[ |\Phi(t)|= \left| \sum_{i=1}^m \lambda_i f_i(0)+ \sum_{k=1}^{n} \alpha_k \Phi^{(k)}(t) \right|< 
 \left|\sum_{i=1}^m \lambda_if_i(0)\right|+\frac{Hn}{8Hn}\leq \frac{1}{4},
\]
which is a contradiction since $t\not\in J$.  For each $1\leq k\leq n$ set 
$$\A_k=\left\{ t\in [a,T]: |\Phi^{(k)}(t)|\geq 1/(8Hn) \right\}.$$
Again by ~\cref{sublevel}, each $\A_k$ is a union of at most  $3n$  intervals and so by ~\cref{interval} there exists  $\kappa_2(n) \le 2(3n)^4$ disjoint intervals $I''_1, \dots, I''_q$ such that 
$\cup_{k=1}^{n} \A_k=\cup_{i=1}^q I''_i,$
and for each $1 \le j \le q$, there exists $1 \le i \le n$ such that 
$I''_j \subseteq \A_i$. Hence from~\cref{partition} we deduce that
$$
[a,T]\setminus J=\bigcup_{i,j} (I'_i\cap I''_j), 
$$
is a union of at most $\kappa_3(n)$ disjoint intervals $I_{ij}:=I'_i\cap I''_j$ such that in each interval $\Phi'(t)$ is monotone and $|\Phi^{(k)}(t)|\geq 1/(8Hn)$ for a fixed $1\leq k\leq n$. 
Hence, by applying van der Corput lemma for each $I_{ij}$, there exists an absolute constant $C_1>0$, depending only on the polynomials $f_1,\dots,f_m$, such that for all $ 1\le i \le n$ and $ 1 \le j \le q$ we have
\begin{equation}\label{IIJ}
\left|  \int_{I_{ij}} \cos (2\pi\Phi(t) ) \d{t}  \right|\leq C_1.
\end{equation}
Moreover since the intervals $I_{ij}$ partition $[a,T] \backslash J$ into $\kappa_3(n)$ intervals, then there exists another absolute constant $C_2>0$, depending only on the polynomials $f_1,\dots,f_m$, such that 
\[ \left|  \int_{[a,T]\setminus J} \cos (2\pi\Phi(t) ) \d{t}  \right|\leq C_2.\]
Also it is immediate that 
$\int_J \cos( 2\pi\Phi(t) ) \d{t} \ge 0,$
since on $J$ we have $|2\pi\Phi(t)|\leq \pi/2$. Therefore 
$$
-C_2/(T-a)\leq\widehat{\mu_T}(\lambda_1,\dots,\lambda_m), 
$$ 
when $\left|\sum_{i=1}^m \lambda_i f_i(0)\right|\leq 1/8$.

\subsection{High frequency case} Now let $(\lambda_1,\dots,\lambda_m)\in\RR^m$ be an arbitrary vector with this property 
\begin{equation}\label{second case}
\left|\sum_{i=1}^m\lambda_i f_i(0)\right|> 1/8. 
\end{equation}
As before, we let $n=\max_i \deg f_i$ and write $f_i(t)=\sum_{j=0}^n a_{ij}t^j$ for $ 1 \le i \le m$. Since $1,f_1,\dots,f_m$ are linearly independent, at least one of the $m\times m$ submatrices of the $m$ by $n$ matrix
$$
\footnotesize{
A'=\begin{pmatrix}
a_{11} & \cdots & a_{1n}\\
\vdots  & \ddots & \vdots\\
a_{m1}  & \cdots & a_{mn}\\ 
\end{pmatrix}
}
$$ 
is invertible. Set 
\begin{equation}\label{A=max}
M:=\max\left\{ \|A^{-1}\|_{\mathrm{op}}:\, \text{$A$ is an invertible $m\times m$ submatrix of $A'$} \right\},\qquad L:=\left(\sum_{i=1}^m f_i^2(0)\right)^{1/2}.
\end{equation}
From~\cref{second case} we have $L>0$. We show that there exists $1\leq \ell\leq n$ such that 
\begin{equation}\label{epsilon}
\left|\sum_{i=1}^m \lambda_i a_{i\ell}\right|> \varepsilon:=\frac{1}{8\sqrt{m}LM}.
\end{equation}
Assume the contrary. Then $(\lambda_1, \dots,\lambda_m) A'\in [-\varepsilon,\varepsilon]^n$.
Applying $A^{-1}$ with $A$ in~\cref{A=max}, we obtain 
\begin{equation}\label{firstbound}
\| (\lambda_1,\dots,\lambda_m) \| \le \|A^{-1} \|_{\mathrm{op}} \sqrt{m} \varepsilon . 
\end{equation}
From the Cauchy--Schwarz inequality applied to~\cref{second case} and \cref{firstbound} we obtain
$$
1/(8L)< (\lambda_1^2+\dots+\lambda_m^2)^{1/2}\leq M\sqrt{m}\varepsilon,
$$
which is a contradiction. We will fix the value of  $\ell$ to satisfy \cref{epsilon} throughout the rest of the argument. 
\begin{lemma}\label{Phi^k-second case}
There exist real numbers $\beta_1, \dots, \beta_n$, where $n=\max_i \deg f_i$, not all zero such that 
\begin{equation}\label{eq 2}
\left(\sum_{i=1}^m\lambda_i a_{i\ell}\right)e^{\ell t} =\sum_{k=1}^{n} \beta_k \Phi^{(k)}(t)
\end{equation}
\end{lemma}
\begin{proof}
Using Lagrange's interpolation method (or the Vandermonde determinant) we find real numbers $\beta_1,\dots,\beta_n$, not all zero, such that 
$
\sum_{k=1}^n \beta_kj^k=\delta_{j\ell}
$, where $\delta_{j\ell}$ is the Kronecker delta.
Hence 
$$
\sum_{k=1}^{n}\beta_k\Phi^{(k)}(t)=\sum_{k=1}^{n}\sum_{i=1}^m\sum_{j=1}^{n} \beta_k j^k \lambda_i a_{ij} e^{jt}=\sum_{i=1}^m\sum_{j=1}^{n} \left(\sum_{k=1}^n \beta_k j^k\right) \lambda_i a_{ij}e^{jt}=\left(\sum_{i=1}^m\lambda_i a_{i\ell}\right)e^{\ell t}.
$$
This finishes the proof of the lemma. 
\end{proof}
Set $H'=\max_{1\leq k\leq n }|\beta_k|$, where $\beta_1, \dots, \beta_n$ are provided by ~\cref{Phi^k-second case}. We claim that for every $t\in [a,T]$, there exists $1 \le k \le n$ such that 
\begin{equation}\label{big-epsilon}
\left| \Phi^{(k)}(t)\right|\geq \varepsilon/(nH'),
\end{equation}
where $\varepsilon$ is defined in~\cref{epsilon}. Assume the contrary. Since $\ell t\geq 0$ we obtain
\begin{equation}
\left|\sum_{i=1}^m\lambda_i a_{i\ell}\right| \leq \left|\left(\sum_{i=1}^m\lambda_i a_{i\ell}\right)e^{\ell t}\right|=  \left|\sum_{k=1}^n \beta_k \Phi^{(k)}(t)  \right|
\le nH'\frac{\varepsilon}{nH'}=\varepsilon, 
\end{equation}
which is a contradiction since $\left|\sum_{i=1}^m\lambda_i a_{i\ell}\right|>\varepsilon$. For each $1\leq k\leq n$, set 
$$
\B_k:=\left\{t\in [a,T]: |\Phi^{(k)}(t)|\geq \varepsilon/nH'\right\}.
$$ 
Since $\sum_{i=1}^m\lambda_i f_i(t)$ is a polynomial of degree at least $1$ and at most $n$, it follows from  ~\cref{sublevel} that
$\B_k$ can be written as a disjoint union of $\kappa_4(n)$ intervals $I'_{k1}, \dots, I'_{kp_k}$ such that $\Phi'(t)$ is monotone on each $I'_i$. Applying ~\cref{interval} to $\{I'_{kj}: 1\leq k\leq n, 1\leq j\leq p_k\}$ we find $\kappa_5(n)$ disjoint intervals $I''_1, \dots, I''_q$ such that 
\begin{equation}\label{B_k-identity}
\bigcup_{k=1}^{n} \B_k=\bigcup_{i=1}^q I''_i,
\end{equation}
and in each interval $I''_i$ the function $\Phi'(t)$ is monotone and $|\Phi^{(k)}(t)|\geq \varepsilon/(H'n)$ for some fixed $1\leq k\leq n$ depending on $i$. 
Hence, by applying van der Corput lemma for each $I''_i$, we can find a constant $C_3>0$, depending only on the polynomials $f_1,\dots,f_m$, such that
\begin{equation}\label{IIJ-2}
\left|  \int_{I''_i} \cos (2\pi\Phi(t) ) \d{t}  \right|\leq C_3.
\end{equation}
Arguing as in the previous case, we can find a constant $C_4>0$, depending only on the polynomials $f_1,\dots,f_m $, such that 
$$
-C_4/(T-a)\leq \widehat{\mu_T}(\lambda_1,\dots,\lambda_m),
$$
when $\left|\sum_{i=1}^m \lambda_i f_i(0)\right|>1/8$. 
This finishes the proof of the second case and thus the proof of ~\cref{lower bound}.  
\end{proof}
By ~\cref{positive-negative} and ~\cref{lower bound} we deduce that
$
-C/(T-a)\leq \inf_{u\in \RR^m}\widehat{\mu_T}(u)<0,
$
for some absolute constant $C>0$. This along with ~\cref{indep-Hoff-spectral} proves ~\cref{Inde-Borel} when $\K=\RR$.

\begin{remark}
Note that the proof in the high frequency case uses the condition $\ell t >0$ 
for $t \in [a, T]$. This is the reason that we have imposed the condition $a_0
\ge 0$. However, if $f_1(0)= \cdots = f_m(0)=0$ and zero is the only common root of $f_i$, then 
we only deal with the low frequency case and $a_0$ can be taken 
to be $- \infty$. 
\end{remark}

\section{Oscillatory integrals with $p$-adic polynomial phase} 
We now turn to the proof of~\cref{Inde-Borel} when $\K=\QQ_p$. Throughout this section the $p$-adic norm is denoted by $\|\cdot\|_p$ and we let $T>a$ be positive integers. Denote the ring of integers of $\QQ_p$  by $\ZZ_p$ and set $\p= p \ZZ_p$. From here we have  
$\p^k=\{s\in \QQ_p: \|s\|_p\leq p^{-k}\}$, and we obtain the filtration 
$$
\QQ_p\supseteq\cdots\supseteq \p^{-2}\supseteq \p^{-1}\supseteq \p^0=\ZZ_p\supseteq\p^1\supseteq \p^2\supseteq\cdots\; . 
$$
For $k\in \ZZ$, denote the sphere of radius $p^k$ by 
$$
\C_k:=\left\{s\in \QQ_p: \|s\|_p=p^k\right\}. 
$$
Notice that $\C_k=\p^{-k}\setminus \p^{-(k-1)}$,
and so $|\C_k|=p^k-p^{k-1}$, where $|\C_k|$ denotes the measure of $\C_k$ with respect to the normalized Haar measure.

\begin{definition}\label{Essentail-part} Let $f(x)=a_n x^n+\cdots+a_1x+a_0\in \QQ_p[x]$ be a polynomial with $a_n\neq 0$ and $n\geq 1$. The {\bf essential part} of $f$ is defined by
\begin{equation}
\mathrm{Ess}_f:=\max\left\{0, \log_p\left(\|a_i\|_p/\|a_n\|_p\right):\,\, 0\neq a_i,\,\, 0\leq i\leq n-1\right\}.
\end{equation}
\end{definition} 
It follows from the ultrametric property of the $p$-adic norm that if $\|s\|\geq p^a$ and $a>\mathrm{Ess}_f$ then
\begin{equation}
\|f(s)\|_p=\|a_ns^n\|_p.
\end{equation}
Let $\psi$ be the Tate character defined by~\cref{additive-char}. 
\begin{lemma}\label{pol-uniform-p-adic}
Let $f(x)\in \QQ_p[x]$ be a polynomial with $\deg f\geq 1$. Then for any $\lambda\in\QQ_p$ and any integers $a$ and $T$ with $T>a>\mathrm{Ess}_f$  we have
\begin{equation}
-16p^{\deg f}\leq \int_{p^a\leq \|s\|_p\leq p^T}\frac{\psi(\lambda f(s))+\bar{\psi}(\lambda f(s))}{\|s\|_p}\, \d{s}.
\end{equation}
\end{lemma}
Similar to the real case, this $p$-adic oscillatory integral can be estimated by a suitable van der Corput lemma for $p$-adic integrals.
\begin{theorem}[$p$-adic van der Corput Lemma]\label{p-adic van der corput}  Suppose $a_0,\dots, a_n\in \QQ_p$, where $a_n\neq 0$ and $n\geq 1$. Then for all $r\in \ZZ$, we have
$$
\left|\int_{\p^r}\psi(a_0+a_1s+\dots+a_ns^n)\, \d{s}\right|\leq \frac{2p^n}{\|a_n\|_p^{1/n}},
$$
\end{theorem}
\begin{proof}
See~\cite[Corollary 5]{Rogers}. 
\end{proof}
Now by exploiting the $p$-adic van der Corput lemma we prove the above lemma. 
\begin{proof}[Proof of ~\cref{pol-uniform-p-adic}] 
Let $f(x)=a_nx^n+\cdots+a_1x+a_0$ be the given polynomial. The inequality is evident if $\lambda=0$. Hence assume that $\lambda\neq 0$ and let $\|\lambda a_n\|_p=p^\ell$ where $\ell\in\ZZ$. We claim that $\psi(\lambda f(s))=1$ whenever $p^a\leq \|s\|_p\leq p^{k_0}$ where $k_0:=\lfloor -\ell/n\rfloor$. To see this, notice that $a>\mathrm{Ess}_f$ and for any $s\in \C_k$ with $a\leq k\leq k_0$ we have 
\begin{equation}
\|\lambda f(s)\|_p=\|\lambda a_n s^n\|_p=p^{\ell+nk}\leq 1,
\end{equation}
which proves the claim since $\ZZ_p=\ker\psi$. Set $k_1=\max\{k_0+1,a\}$. Therefore 
\begin{equation}\label{incomplite-ine}
\int_{p^{k_1}\leq \|s\|_p\leq p^T}\frac{\psi( \lambda f(s))+\bar{\psi}(\lambda f(s))}{\|s\|_p}\, \d{s}\leq\int_{p^a\leq \|s\|_p\leq p^T}\frac{\psi(\lambda f(s))+\bar{\psi}(\lambda f(s))}{\|s\|_p}\, \d{s}.
\end{equation}
It thus suffices to find a lower bound for the left hand side of the above inequality. We have
\begin{equation}\label{p1}
\begin{split}
\left|\int_{p^{k_1}\leq \|s\|_p\leq p^{T}}\frac{\psi(\lambda f(s))+\bar{\psi}(\lambda f(s))}{\|s\|_p}\, \d{s}\right|&\leq\sum_{r=k_1}^{T} \left|\int_{\mathcal{C}_r}\frac{\psi(\lambda f(s))+\bar{\psi}(\lambda f(s))}{\|s\|_p}\, \d{s}\right|\\
&=\sum_{r=k_1}^{T} \frac{1}{p^r}\left|\int_{\mathcal{C}_r}(\psi(\lambda f(s))+\bar{\psi}(\lambda f(s)))\, \d{s}\right|\\
&\leq 2\sum_{r=k_1}^{T} \frac{1}{p^r}\left|\int_{\mathcal{C}_r}\psi(\lambda f(s))\, \d{s}\right|.
\end{split}
\end{equation}
Notice that $\mathcal{C}_r=\p^{-r}\setminus \p^{-(r-1)}$. Hence by~\cref{p-adic van der corput} we obtain
\begin{equation}\label{p2}
\left|\int_{\mathcal{C}_r}\psi(\lambda f(s))\d{s}\right|=\left|\int_{\p^{-r}}\psi(\lambda f(s))\, \d{s}-\int_{\p^{-(r-1)}}\psi(\lambda f(s))\, \d{s}\right|\leq \frac{4p^n}{\|\lambda a_n\|_p^{1/n}}.
\end{equation}
From~\cref{p1} and~\cref{p2} we conclude that 
\begin{equation*}
\left|\int_{p^{k_1}\leq \|s\|_p\leq p^T}\frac{\psi(\lambda f(s))+\bar{\psi}(\lambda f(s))}{\|s\|_p}\, \d{s}\right|\leq \frac{8p^n}{p^{\ell/n}}\sum_{r=p^{k_1}}^T\frac{1}{p^r}\leq \frac{8p^n}{p^{\ell/n+k_1}}\sum_{r\geq 0}\frac{1}{p^r}\leq 16p^{n},
\end{equation*}
since $\ell/n+k_1\geq 0$. This inequality along with~\cref{incomplite-ine} provide the lower bound.
\end{proof}
Now we return to the proof of~\cref{Inde-Borel}. 
In the course of the proof,
we will need the following general fact. Suppose $S \subseteq \QQ_p^m$ is Borel 
and $E \in \GL_m(\QQ_p)$. Then 
\begin{equation*}
 \frac{|S \cap B_r| }{|B_r|}
= \|\det E \|^{-1}_p    \frac{|E(S) \cap E(B_r)| }{|B_r|}
\le \|\det E \|^{-1}_p  \frac{|E(S) \cap B_{ r \| E  \|_{\op}   }   |   }{|B_r|} 
 = \frac{ \|\det E \|^{-1}_p}{ \| E   \|_{\op}^{-m} }  \frac{|E(S) \cap B_{ r \| E  \|_{\op} } |}{ |B_{ r \|   E  \|_{\op}   }| },
\end{equation*}
where $ \| E \|_{\op}:= \max_{ 1 \le i, j \le m} \left\|  E_{ij} \right\|_p$. This clearly implies that
\begin{equation}\label{trivial}
\frac{\|\det E \|_p}{ \| E  \|_{\op}^{m} } \ \overline{d}( S) \le  \overline{d}(E(S)).
\end{equation}

 Similar to the real case, let  $f_1(x),\dots, f_m(x)$ be a family of $m\geq 2$ polynomials in $\QQ_p[x]$. Define the following symmetric, bounded Borel set
\begin{equation}\label{SC-padic}
S:=S_{a,T}=S_{a,T,f_1,\dots,f_m}=\left\{\pm (f_1(s),\dots, f_m(s))\in \QQ_p^m:  p^a\leq \|s\|_p\leq p^T\right\}. 
\end{equation}
Let $n=\max_{1\leq i \le  m}\deg f_i$ and $f_i(x)=\sum_{j=0}^n a_{ij}x^j$. 
Since, by our assumption, $1,f_1,\dots,f_m$ are linearly independent, the rank of the coordinate matrix of $f_1,\dots,f_m,1$ with respect to the ordered basis $\{x^n,\dots,x,1\}$ given by
$$
\footnotesize
A=\begin{pmatrix}
a_{1,n}         &     a_{1,n-1}           & \cdots  &  a_{1,1}  & a_{1,0}\\
\vdots  &       \vdots           &  &   \vdots & \vdots\\
a_{m,n}         &     a_{m,n-1}           & \cdots  &  a_{m,1}&   a_{m,0}\\
0         &    0          & \cdots  &  0&   1
\end{pmatrix},
$$
is $m+1$. Thus by applying the Gauss elimination method,
$A$ can be reduced to a matrix in the row echelon form of rank $m+1$ with the same last row as $A$. Hence there exists $B\in \GL_m(\QQ_p)$ such that 
$$
B(f_1,\dots,f_m)^{\mathrm{Tr}}=(f'_1,\dots,f'_m)^{\mathrm{Tr}},
$$
with $\deg f'_1>\deg f'_2>\dots>\deg f'_m\geq 1$.

By applying \eqref{trivial} to the linear transformation $B\in\GL_m(\QQ_p)$, there exists $c>0$ depending only on $f_1,\dots,f_m$ such that
$$
c\bar{\alpha}(\Cay(\QQ_p^m,S) ) \leq \bar{\alpha}(\Cay(\QQ_p^m, B(S)))=\bar{\alpha}(\Cay(\QQ_p^m,S_{a,T,f'_1,\dots, f'_m})).$$

So we may and will assume from now on that $\deg f_1>\deg f_2>\dots>\deg f_m\geq 1$.  We will also assume that 
$a$ and $T$ are positive integers with 
$$T>a >a_0:=\max\left\{\mathrm{Ess}_{f_i}: 1\leq i\leq m\right\},$$
in~\cref{SC-padic} and so $0\not\in \overline{S}$, the closure of $S$.  
Let $\mu_T$ be the measure on $\QQ_p^m$ defined for every Borel subset $E \subseteq \QQ_p^m$ via 
$$
\mu_T(E)=\frac{1}{L}\int_{p^{a}\leq \|s\|_p\leq p^{T}} \frac{\1_E(f_1(s),\cdots,f_m(s))+\1_{(-E)}(f_1(s),\cdots,f_m(s))}{\|s\|_p}\, \d{s},
$$   
where 
$$L=2(T-a+1)\left(1-\frac{1}{p}\right).$$
It is easy to verify that $\mu_T$ is a symmetric Borel probability measure on $\QQ_p^m$ and $\mathrm{Supp}(\mu_T) =S$.  By a straightforward calculation we  obtain 
\begin{equation*}
\widehat{\mu_T}(\lambda_1,\dots,\lambda_m)=
\frac{1}{L}\int_{p^{a}\leq \|s\|_p\leq p^{T}}\left(\psi(\lambda_1 f_1(s)+\cdots+\lambda_m f_m(s))+\overline{\psi}(\lambda_1 f_1(s)+\cdots+\lambda_mf_m(s))\right)\frac{\d{s}}{\|s\|_p},
\end{equation*} 
where $(\lambda_1,\dots,\lambda_m)\in \QQ_p^m$. Our aim now is to estimate $\widehat{\mu_T}(\lambda_1,\dots,\lambda_m)$.
\begin{theorem}\label{p-adic-measure-ine}
For any integer $T>a>a_0$ and arbitrary $(\lambda_1,\dots,\lambda_m)\in \QQ_p^m$ we have 
\begin{equation}\label{vector-ine-padic}
-\frac{16\sum_{i=1}^m p^{\deg f_i}}{L}\leq \widehat{\mu_T}(\lambda_1,\dots,\lambda_m).
\end{equation}
\end{theorem}
\begin{proof}[Proof of ~\cref{p-adic-measure-ine}]
 ~\cref{pol-uniform-p-adic} establishes~\cref{vector-ine-padic} when $m=1$. We prove the theorem by induction on $m$. Let $f_i(x)=\sum_{j=0}^{n_i} a_{ij}x^j$ and set 
$$
\Psi_{f_1,\dots,f_m}(s):=\psi(\lambda_1 f_1(s)+\cdots+\lambda_m f_m(s))+\overline{\psi}(\lambda_1 f_1(s)+\cdots+\lambda_mf_m(s)).
$$
By the induction hypothesis we can assume that $\lambda_1\neq 0$.  Assume that $\|\lambda_1a_{1n_1}\|_p=p^\ell$, where $\ell\in\ZZ$. Set $k_0:=\lfloor-\ell/n_1\rfloor$. Then, similar to the proof of~\cref{pol-uniform-p-adic}, we have 
\begin{equation}\label{psi-1}
\psi(\lambda_1 f_1(s))=1,\qquad p^a\leq \|s\|_p\leq p^{k_0}.
\end{equation}
Set $k_1=\max\{k_0+1,a\}$. Thus from~\cref{psi-1} and the induction hypothesis we obtain
\begin{equation}\label{ind-equ}
\begin{split}
\int_{p^a\leq \|s\|_p\leq p^T}\frac{\Psi_{f_1,\dots,f_m}(s)}{\|s\|_p}\, \d{s}&=\int_{p^a\leq \|s\|_p\leq p^{k_1-1}}\frac{\Psi_{f_2,\dots,f_{m}}(s)}{\|s\|_p}\, \d{s}+\int_{p^{k_1}\leq \|s\|_p\leq p^T}\frac{\Psi_{f_1,\dots,f_m}(s)}{\|s\|_p}\, \d{s}\\
&\geq -16\sum_{2\leq i\leq m}p^{\deg f_i}+\int_{p^{k_1}\leq \|s\|_p\leq p^T}\frac{\Psi_{f_1,\dots,f_m}(s)}{\|s\|_p}\, \d{s}.
\end{split}
\end{equation} 
Note that since $\deg f_1> \cdots > \deg f_m\geq 1$, as in ~\cref{pol-uniform-p-adic}, we apply the $p$-adic van der Corput's lemma to obtain 
\begin{equation}
\begin{split}
\left| \int_{p^{k_1}\leq \|s\|_p\leq p^T}\frac{\Psi_{f_1,\dots,f_m}(s)}{\|s\|_p}\, \d{s}\right|&\leq \sum_{r=k_1}^T\frac{1}{p^r}\left| \int_{\mathcal{C}_r}\Psi_{f_1,\dots,f_m}(s)\, \d{s}\right|\leq\frac{8p^{n_1}}{\|\lambda_1 a_{1n_1}\|^{1/{n_1}}_p}\sum_{r=k_1}^\infty\frac{1}{p^r}\leq 16p^{n_1}.
\end{split}
\end{equation}
This along with~\cref{ind-equ} proves ~\cref{p-adic-measure-ine}.
\end{proof}
Now by ~\cref{positive-negative} and ~\cref{p-adic-measure-ine} we deduce that
$
-C/(T-a)\leq \inf_{u\in \QQ_p^m}\widehat{\mu_T}(u)<0,
$
for some absolute constant $C>0$. This along with ~\cref{indep-Hoff-spectral} proves ~\cref{Inde-Borel} for when $\K=\QQ_p$. 
\section{Clique number and affine B\'{e}zout theorem}
In this section, we will prove~\cref{clique}. Before starting the proof, we will recall the following affine version of B\'{e}zout's theorem. We assume that $\KK$ is an algebraic closed field of characteristic zero.  
\begin{theorem}\label{Bezout} Let $V_1, V_2\subseteq \KK^m$ be two one-dimensional irreducible varieties defined by polynomials $P_1$,$\dots$, $P_r$ with $\deg P_i=d_i$ (respectively $P'_1$,$\dots$, $P'_s$ with $\deg P'_i=d'_i$). Then either $V_1=V_2$ or 
$$|V_1\cap V_2|\leq d_1\dots d_r d'_1\dots d'_s.$$
\end{theorem}
\begin{proof}
Notice that $V_1\cap V_2$ is a finite set and $r+s\geq m$. Hence by~\cite[Theorem 3.1]{Schmid} we obtain the inequality. For more details see also~\cite[Theorem 5]{Tao}. 
\end{proof}
Recall that the Ramsey number $R(d,d)$ is the least integer $m$ such that for any edge coloring of the complete graph $K_m$ in red and blue, there exist $d$ vertices forming a monochromatic $K_d$. 
\begin{proof}[Proof of~\cref{clique}]
Let $V\subseteq {\KK}^m$ be an irreducible variety of dimension $1$ that is not an affine line. Moreover assume that $V$ is defined by $P_1= \cdots = P_n=0$, where $P_i\in \KK[x_1,\dots,x_m]$ with $d_i=\deg P_i$. Set $d=(d_1\dots d_n)^2+3$ and $m=R(d,d)$. Suppose $\G=\Cay(\KK^m,\pm V)$ contains a copy of the complete graph $K_m$ formed with vertices $u_1,\dots, u_m$. For $1\leq i<j\leq m$, color the edge between $u_i$ and $u_j$ red if $u_j-u_i\in V$, and blue if $u_j-u_i \in (-V) \backslash V$. In view of $m=R(d,d)$, there exists a monochromatic complete graph on $d$ vertices. Without loss of generality, we will assume that
it is a red $K_d$ with vertices $u_1, \dots, u_{d}$, thus $u_j-u_i \in V$ for $j >i$. 
Consider the variety $V'= V+(u_2-u_1)$. Note that for $3 \le i \le d$, we have $u_i - u_1 \in V$ and
$u_i - u_1=  ( u_i-u_2)+ (u_2-u_1) \in V+(u_2-u_1)=V'$. This implies that $V$ and $V'$ have at least $d-2>(d_1\dots d_n)^2$ 
intersection points. Hence by~\cref{Bezout} we have $V=V+(u_2-u_1)$. If $v$ is an arbitrary point on $V$, it follows that $V$ contains the points $v-n( u_2-u_1)$ for all integers $n \ge 1$. Hence $V$ contains the Zariski closure of these points, which is a line. Since $V$ is irreducible, it follows that $V$ is a line, which is contrary to the hypotheses. Therefore $\omega(\Cay(\KK^m,\pm V))< m$. 
\end{proof}

\section{Some remarks on Cayley graphs of curves }
In this section, we will investigate the question of coloring for Cayley graphs with respect to curves
other than the ones studied above. For simplicity, we will restrict the discussion to the case $n=2$. Extension to the general case is straightforward. Let $F: \RR \to \RR^2$ be a continuous function. Set  $$
\G_F:=\Cay(\RR^2, S_F),\qquad S_F=\{\pm F(s): s\in \RR\}.
$$

\begin{remark}
For $b>a>0$, and a linear functional $\phi: \RR^2 \to \RR$, let $H_{a,b}$ denote the region defined by 
$a < \phi(x) < b$. Hence, if $ \phi(F(t)) \in (a,b)$, then
$$\bchi(\G_F) \le \bchi(\Cay(\RR^2,  \pm H_{a,b})) \le \bchi(\Cay(\RR,  \pm (a,b))< \infty.$$
\end{remark}

\begin{proposition}\label{Cos} Let $f: \mathbb{R}\to \mathbb{R}$ be a continuous periodic function with $f(0)\neq 0$. If $F(t)=(t,f(t))$, then $\bchi(\G_F)< \infty$. 
\end{proposition}
\begin{proof} Without loss of generality we can assume that the period of the function is $1$. Since $f(0)\neq 0$ then for some $\varepsilon>0$ there exits $\delta>0$ such that if $|x|\leq \delta$ then $\varepsilon<|f(x)|$. Assume $\|f\|_\infty\leq M$ and pick an integer $n> \max(M+2,\epsilon^{-1}, \delta^{-1})$. Define the following Borel measurable function
$$
c: \mathbb{R}^2\to \mathbb{Z}/(n \mathbb{Z})\times \mathbb{Z}/n^2\mathbb{Z},\qquad (x,y)\mapsto (\lfloor nx \rfloor \pmod{n }, \lfloor ny\rfloor \pmod{n^2}).
$$ 
Suppose $c(x,y)=c(x',y')$. Therefore $\lfloor ny'\rfloor=n^2k+\lfloor ny\rfloor$ for some $k\in \mathbb{Z}^{\geq 0}$. First assume $k>0$. Then $|y'-y|>M$ and so $(x,y)$ and $(x',y')$ do not form an edge. Hence assume $k=0$ and so 
\begin{equation}\label{y}
|y'-y|<1/n<\varepsilon
\end{equation}
Moreover we have $\lfloor nx'\rfloor=n k_1+\lfloor nx\rfloor$ for some $k_1\in\ZZ$. Let 
$nx=\lfloor nx\rfloor+\theta$ and $nx'=\lfloor nx'\rfloor+\theta'$, 
where $\theta,\theta'\in [0,1)$. This implies that
$|x-x'-k_1|\le 1/n< \delta$. By periodicity, $|f(x'-x)|> \epsilon$.
From here and~\cref{y} we deduce that  $(x,y)$ and $(x',y')$ can not be an edge.
\end{proof}

Note that the graph of a continuous periodic function is also bounded by two hyperplanes. 
In the rest of section, we will prove ~\cref{analytic}

\begin{proof}[Proof of  ~\cref{analytic}]
Let $ \Lambda=\{q_n=3^{n-1}: n\geq 1\}$. Then $\chi(\Cay(\ZZ,\pm\Lambda))=2$ with the color classes $2\ZZ$ and $2\ZZ+1$. Let $ c: \ZZ \to \{1,2\}$ be the given coloring map.
Let $I= [-1/2, 1/2)$, and for $x \in I$, define $\omega(x)=-1$ for $x<0$ and $\omega(x)=1$
for $x \ge 0$. For every $x \in \RR$, write $x= z(x)+h(x)$ with $z(x) \in \ZZ$ and $h(x) \in I$. Consider the map
\[  \widetilde{ c} (x)=    (c( z(x)), \omega(  h(x)) ). \]
We claim that if $ \widetilde{ c}(x_1)=  \widetilde{  c}(x_2)$ then for all $n \ge 1$ we have $| x_1-x_2- q_n|> 1/4$. Assume the contrary, that is,  $ \widetilde{ c}(x_1)=  \widetilde{  c}(x_2)$ and that
$| x_1-x_2- q_n| \le 1/4.$ It follows that
\begin{equation*}
\begin{split}
 |z(x_1)-z(x_2)- q_n|&= |(x_1-x_2-q_n)-(h(x_1)-h(x_2))|\\
 &\le |x_1-x_2-q_n|+|h(x_1)-h(x_2)|
  \le 1/4+1/2 <1.
 \end{split}
\end{equation*}
This implies that $z(x_1)-z(x_2)=q_n$, which is a contradiction to $ c( z(x_1))= c( z(x_2))$. Now, consider the region $S$ defined by
\[  S= \{ (x, \pm q_n+y ): x \in \RR, |y|< 1/4 \} \cup \{ (\pm q_n + x, y): y \in \RR, |x|< 1/4 \}. \]
We claim that $\Cay(\RR^2, S)$ has a finite Borel chromatic number. In fact, for
each $(x,y) \in \RR^2$, set $ c_2(x, y)= ( \widetilde{ c}(x), \widetilde{ c}(y))$. Suppose
$(x_1,y_1), (x_2,y_2) \in \RR^2$ are such that $(x_1-x_2, y_1-y_2) \in S$.
Then there exists $n \ge 1$ such that
$|x_1-x_2 -q_n|< 1/4$ or $|y_1 - y_2- q_n|< 1/4$. In the former case,
we have  $ \widetilde{ c}(x_1) \neq \widetilde{ c}(x_2)$ and in the latter case, we have
$ \widetilde{c}(y_1) \neq \widetilde{ c}(y_2)$. From this we conclude that $c_2(x_1,y_1)\neq c_2(x_2,y_2)$.

It now remains to see that $S$ contains the graph of an analytic function which is not bounded between any two parallel lines. Define the piecewise linear echelon-shaped curve $L$ by
\[ L= \bigcup_{n=1}^{\infty} \{ (x,q_n): q_{n} \le x \le q_{n+1} \}
\cup   \bigcup_{n=1}^{\infty} \{ (q_{n+1},y): q_{n} \le y \le q_{n+1} \}. \]
Note that $L$ does not lie in the region between any two parallel lines. Identify $\RR^2$ with the complex plane. Let $ S_1$ be the $ \delta$-neighborhood of $L$ for $ \delta< 1/4$.  It is easy to see that
$S_1 \subset S$ and $S_1$ is a simply connected domain whose boundary viewed as a subset of the Riemann sphere
is a Jordan curve containing $\infty$. Let $\mathbb{D}= \{ z  \in \CC: |z|<1 \}$ denote the open unit disk in $\CC$. Denote by $\phi: \mathbb{D} \to S_1$ the biholomorphic map provided by the Riemann mapping theorem. Recall that Caratheodory's extension of Riemann mapping theorem (see e.g. Theorem 5.1.1 in \cite{Krantz}) asserts that if $ \Omega$ is a bounded simply connected domain in $\CC$ whose boundary is a single Jordan curve, then the Riemann conformal map $\phi_1: \mathbb{D} \to \Omega$ extends continuously to a continuous one-to-one function $ \widehat{\phi_1}: \overline{\mathbb{D}} \to \overline{ \Omega}$.
Note that $S_1$ is not bounded, and hence Caratheodory's theorem is not directly applicable. Since $S_1$ is included in the region $\Re z >0$ and $\Im z >0$, for all $z \in S_1$, we have $|z -(1-i)| \ge 1$. Setting $\psi(z) = \frac{1}{z-(1-i)}$, it follows that $\phi_1:=\psi \circ \phi$ maps $\mathbb{D}$ to a bounded domain $ \psi(S_1)=\Omega$ whose boundary is a Jordan curve passing through $\psi(  \infty)=0$.  

Applying Caratheodory's theorem, we obtain a continuous map $ \widehat{\phi_1}: \overline{\mathbb{D}} \to \overline{ \Omega}$, whose restriction to the boundary $ \partial \overline{\mathbb{D}}$ is a homeomorphism, and its restriction to $\mathbb{D}$ is holomorphic. Thus there exists $r \in \partial\mathbb{D}$ such that $ \widehat{\phi_1}(r)= 0$. After possibly precomposing $ \widehat{\phi_1}$ with a rotation, we can assume that $r=1$. Define $h: (0,1) \to S_1$ by $h(t)=\psi^{-1}(\phi_1(t))$. Since $h(1)= \infty$,  and the image of $h$ lies in $S_1$, the function $h$ provides the required real analytic curve.
\end{proof}
\begin{remark}
The above construction has the following generalization. Let $\Lambda=\{q_n\}_{n\geq 1}$ be an increasing lacunary sequence, that is, $q_{n+1}/q_n\geq 1+\varepsilon$ for some $\varepsilon> 0$. In~\cite{Katznelson}, Katznelson proved that under these
hypotheses we have $\chi(\Cay(\ZZ,\pm\Lambda))<\infty$. By applying this theorem, the set $\{3^{n-1}: n\geq 1\}$ in the above proof, can be replaced by any lacunary sequence $\Lambda$.  
\end{remark}
It would be desirable to obtain an adequate description of the sets $\SC\subseteq \RR^m$ of Hausdorff dimension at least one for which the Borel chromatic number of $\Cay(\RR^m, \SC)$ is infinite. In particular, it would be interesting to know the answer to  
the following two questions:
\begin{question}
Suppose $\SC$ is an irreducible Zariski closed subset of $ \CC^m$ which is not contained in an affine hyperplane. Is it true that 
$\bchi(\Cay(\CC^m, \pm\SC))$ is infinite? Note that $\SC$ is automatically non-compact. 
\end{question}
\begin{question}
Suppose $\SC_0$ is an irreducible Zariski closed subset of $ \RR^m$ which is not contained in an affine hyperplane. Let $d_i$ tends to $\infty$
and $\SC= \bigcup_{n=1}^{\infty} d_n \SC_0$ be a union of dilations of $\SC_0$. Is it true that $\bchi(\Cay(\RR^m, \pm\SC))$ is infinite?
\end{question}
\section{Algebraic aspects of Cayley graphs} 
This short section is devoted to the proof of~\cref{Cass-imb}. The following result due to Cassels~\cite{Cassels} will be needed in the proof. 
\begin{theorem}\label{Ca}
Let $K$ be a finitely generated extension of $\QQ$ and 
$C \subseteq K \setminus \{ 0 \}$ be a finite set. Then there exist
infinitely many primes $p$ for which there exists an embedding
$\iota: K \to \QQ_{p}$ such that $\iota(x) \in \ZZ_p \setminus p \ZZ_p$ for all $x \in C$. 
\end{theorem}
Moreover we recall the following theorem of de Bruijn and Erd\H{o}s~\cite[Theorem 8.1.3]{Diestel}.
\begin{theorem}\label{debruijn} Let $\G=(V,E)$ be a graph and $k\in \mathbb{N}$. If every finite subgraph of $\G$ has chromatic number at most $k$, then so does $\G$.
\end{theorem}
We now are ready to prove  ~\cref{Cass-imb}. 
 \begin{proof}[Proof of ~\cref{Cass-imb}] 
Recall that $V$ defined by $F_1= \cdots= F_n=0$ is an algebraic set defined over $\ZZ$. For every prime $p$, the algebraic closure of $\QQ_{p}$ is an algebraic closed field of characteristic zero and transcendental degree $2^{ \aleph_0}$ over $\QQ$, and 
hence, by a well-known theorem of Steinitz, is isomorphic to $\CC$. This yields a ring embedding $i: \QQ_{p} \hookrightarrow \CC$, implying that
$$
\sup_{p}\chi\left(\Cay(\ZZ_p^m, \pm V(\ZZ_p))\right) \leq \chi\left(\Cay(\CC^m, \pm V(\CC))\right).
$$
Note that the embedding $i$ is purely algebraic and is far from being continuous or even measurable. 

To prove the reverse inequality, by applying ~\cref{debruijn}, it suffices to show that for any finite subgraph $ \G$ of 
$\Cay(\CC^m, \pm V(\CC))$, there exists a prime $p$ such that 
$ \chi(\G) \le \chi(\Cay(\ZZ_p^m, \pm V(\ZZ_p)))$. Assume that $U \subseteq \CC^m$ denotes the vertex set of $\G$. For $v,w \in U$ forming an edge $vw\in E(\G)$, we have 
\[ F_j( v- w)=0 \qquad \text{or}\qquad F_j(w-v)=0. \]
For any two distinct vertices $v,w \in U$ there exists a vector $\z_{v,w} \in \CC^m$ such that 
\[ \langle v- w ,  \z_{v,w} \rangle =1,  \]
where $\langle \ , \  \rangle $ denotes the standard symmetric bilinear form on $\CC^m$. Denote by $C$ the set of all entries of $v$ for all $ v \in V(\G)$ 
and $\z_{v, w}$ for all pairs of vertices $v,w$. Let $K$ be the finitely generated extension of $\QQ$ generated by $C$.
Note that, viewed as vertices of $\Cay(K^m, \pm V(K))$, the points $v$ span a subgraph isomorphic to $ \G$.  
We now apply ~\cref{Ca} to find a prime $p$ and a field embedding 
$\iota: K  \hookrightarrow \QQ_{p}$, such that $\iota(C) \subseteq \ZZ_{p}$. 
As $\iota$ is a field embedding, we have 
\[ F_j( \iota(v)- \iota(w))=0 \qquad\text{or}\qquad F_j( \iota(w)-\iota(v))=0, \]
for all $1 \le j \le n$ and $vw \in E(\G)$.
Moreover since $\langle\iota(v)- \iota(w), \iota(\z_{v,w})\rangle=1$, it follows that 
$\iota(v) \neq \iota(w)$. Thus $\G$ embeds in $\Cay(\ZZ_p^m, \pm V(\ZZ_p))$ which finishes the proof of the theorem.  
\end{proof}
\section*{Acknowledgement} We would like to thank Martin Bays, Emmanuel Breuillard, Sam Drury, Dmitry Jakobson and Omid Hatami for several useful discussions. We would like to acknowledge our gratitude to James Wright who drew our attention to
~\cite{Nagel}. The authors are indebted to the referee for carefully reading the paper and providing many valuable comments that improved the exposition of the paper. During the completion of this work, M.B was supported by University of M\"{u}nster and University of Cambridge and M.B would like to acknowledge that, this project has received funding from the European Research Council (ERC) under the European Union's Seventh Framework Programme, FP7/2007-2013 (grant agreement No 617129). During the course of this research, K.M-K is partially supported by the DFG grant DI506/14-1.

\bibliographystyle{abbrv}

\end{document}